\documentclass[a4paper,reqno]{amsart}
\usepackage{amsmath, amssymb, eucal, amscd, amstext, verbatim, enumerate, color}

\newtheorem{thm}{Theorem} %[section]
\newtheorem{lem}[thm]{Lemma}

\newtheorem{prop}[thm]{Proposition}
\newtheorem{cor}[thm]{Corollary}

\newtheorem{rem}[thm]{Remark}
\newenvironment{remark}{\begin{rem}\rm}{\end{rem}}
\newtheorem{ex}[thm]{Example}
\newenvironment{example}{\begin{ex}\rm}{\end{ex}}

\newcommand{\CB}{\mathcal{B}}

\newcommand{\KH}{\mathfrak{H}}

\newcommand{\KI}{\mathfrak{I}}

\newcommand{\KK}{\mathfrak{K}}

\newcommand{\BC}{\mathbb{C}}

\newcommand{\BG}{\mathbb{G}}

\newcommand{\supp}{\mathrm{supp}\ \!}
\newcommand{\tI}{\mathrm{I}}
\newcommand{\ru}{\mathrm{u}}

\begin{document}
\title{Property $(T)$ for locally compact groups and $C^*$-algebras}

\author{Bachir Bekka and Chi-Keung Ng}

\address[Bachir Bekka]{ Univ Rennes, IRMAR--UMR 6625, F-35000  Rennes, France.}
\email{bachir.bekka@univ-rennes1.fr}

\address[Chi-Keung Ng]{Chern Institute of Mathematics and LPMC, Nankai University, Tianjin 300071, China.}
\email{ckng@nankai.edu.cn; ckngmath@hotmail.com}



\thanks{The first author is supported  by the  Agence Nationale de la Recherche (ANR-11-LABX-0020-01, ANR-14-CE25-0004); the second author is supported by the National Natural Science Foundation of China (11471168).}
\begin{abstract}
Let $G$ be a  locally compact group  and let   
$C^*(G)$ and  $C^*_r(G)$ be the full group $C^*$-algebra and the reduced group $C^*$-algebra  of $G$.
 We investigate the relationship between Property $(T)$  for $G$ and Property $(T)$ as well as its strong version   for $C^*(G)$ and $C^*_r(G)$ (as defined in \cite{Ng-nu-T}).
We show that  $G$ has Property $(T)$ if (and only if) $C^*(G)$ has Property $(T)$. 
In the case where $G$ is a locally compact IN-group, we prove that  $G$ has Property $(T)$ if and only if $C^*_r(G)$ has 
strong Property $(T)$.  We also show that  $C^*_r(G)$ has  strong Property $(T)$ for every non-amenable locally compact group $G$  for which $C^*_r(G)$ is nuclear. Some of these groups (as for instance $G=SL_2(\mathbf{R})$)  do not have Property $T$.
\end{abstract}

\date{\today}
\maketitle

\section{Introduction and statement of the results}
\label{S0}
Kazhdan's Property $(T)$,  introduced by Kazhdan, in \cite{Kaz}, is a rigidity property of unitary representations of locally compact groups
 with amazing  applications,   ranging  from geometry and group theory to operator algebras and graph theory (see \cite{BHV}).
Property $(T)$ was defined in the context of operator algebras by Connes and Jones, in \cite{ConnesJones}, for von Neumann algebras,
and by the first named author, in \cite{B}, for \emph{unital}  $C^*$-algebras (in terms of approximate central vectors and central vectors
for appropriate bimodules over the  algebras).

This allowed to characterize Property $(T)$ for a \emph{discrete} group $\Gamma$  in terms 
of various operator algebras attached to it: $\Gamma$ has Property $(T)$ if and only if  $A$
has Property $(T)$, where 
$A= C^*_r(\Gamma)$ is
the reduced group $C^*$-algebra, or  $A= C^*(\Gamma)$ is the full group $C^*$-algebra
of $\Gamma$. 
More generally,  Property $(T)$ for a pair of groups $\Lambda \subset \Gamma$ 
 (also called relative Property $(T)$)  is characterized by  Property $(T)$ for  the pair of the corresponding 
 $C^*$-algebras (as for instance $C^*_r(\Lambda)\subset C^*_r(\Gamma)$).
 
Property (T) for unital  $C^*$-algebras  was further studied by various authors (see e.g. \cite{Bro}, \cite{LN}, \cite{LNW} and \cite{Suz}) and  a stronger version of it, called the  strong Property $(T)$, was defined  by the second named author \emph{et al.} in \cite{LN}. 

Let $G$ be a locally compact group. Recall that the  reduced group $C^*$-algebra  $C^*_r(G)$ or the full group $C^*$-algebra  $C^*(G)$ of $G$ is unital 
if and only if $G$ is discrete (see \cite{Milnes}). It is important to study Property (T) in the context of  non-discrete locally compact groups for its own sake or as a tool towards establishing this property for suitable discrete subgroups of them
(this is for instance how usually lattices in higher rank Lie groups are shown to have Property (T); see \cite{BHV}).
So, it is of interest to study the relationship between  Property (T) for $G$ and  suitable rigidity properties 
of $C^*_r(G)$ and  $C^*(G)$.

There have been attempts to define  a notion of Property $(T)$ for  \emph{non-unital} $C^*$-algebras.  
As shown in \cite{LNW},  the  most straightforward extension of the definition given in  \cite{B}    leads to a disappointing result:
no separable non-unital $C^*$-algebra has such Property $(T)$.
In \cite{Ng-nu-T}, the second named author introduced a refined notion of Property $(T)$ and strong Property $(T)$ for general $C^*$-algebras (as well as their relative versions)  which seems to be more sensible (see Section~\ref{S1} below).  
For instance, it was shown in \cite{Ng-nu-T} that, if $G$ has Property $(T)$,  then $C^*(G)$ has strong Property $(T)$. 

The aim of the article is to  further investigate Property $(T)$   as defined in  \cite{Ng-nu-T}.
Our first theorem shows that Property $(T)$ of a general locally compact group is indeed characterized by  Property $(T)$ of its full group $C^*$-algebra.

\begin{thm}
\label{thm:full}
Let $G$ be a locally compact group and $H$ a closed subgroup of $G$. 

\noindent
(a) The pair  $(G,H)$ has Property $(T)$ if and only if 
the pair  $(C^*(G),C^*(H))$ has Property $(T)$. 

\noindent
(b) $G$ has Property $(T)$ if and only if $C^*(G)$ has Property $(T)$ (equivalently, strong Property $(T)$). 
\end{thm}
The proof of Theorem~\ref{thm:full} is based on   an isolation property of 
one dimensional representations (see Proposition~\ref{prop:prop-T>isolate}) of an arbitrary $C^*$-algebra  with Property $(T)$, which is of independent interest.
This isolation property also allows us 
 to give a sufficient condition for a locally compact quantum group to  satisfy Property $(T)$ (see, e.g., \cite{CN}).
Notice that we do not know whether this sufficient condition is also a necessary condition.

\begin{thm}\label{thm:full-QG}
If $\BG$ is a locally compact quantum group such that its full group  $C^*$-algebra $C_0^\ru(\widehat{ \BG})$ has Property $(T)$, then $\BG$ has property $(T)$.
\end{thm}

It is natural to ask whether Theorem \ref{thm:full} remains true when $C^*(G)$ is replaced by $C^*_r(G)$. 
As mentioned above, this is the case when $G$ is discrete.  This also holds when $G$ is amenable: indeed, in this case,
 $C^*_r(G)\cong C^*(G)$ and Theorem \ref{thm:full} shows that $C^*_r(G)$ has Property $(T)$ if and only if $G$ has Property $(T)$, that is (see Theorem 1.1.6 in \cite{BHV}), if and only if $G$ is compact.

However, as we now see, Property $(T)$ of the reduced group $C^*$-algebra does not imply Property $(T)$ of the original group, for a large class of  locally compact groups related to nuclear  $C^*$-algebras (see e.g. \cite{BrownOzawa}).

 \begin{thm}
\label{thm:2}
Let $G$ be a  locally compact group such that  $C^*_r(G)$ is a nuclear $C^*$-algebra.  
 If $G$ is non-amenable, then $C^*_r(G)$ has strong Property $(T).$ 
 \end{thm}
\begin{example}
 \label{Exa-NuclearT}
 (i) The class of locally compact groups $G$ for which $C^*_r(G)$ is a nuclear $C^*$-algebra
 includes not just all amenable groups but also all connected groups  as well as  all groups of type $\tI$  (see \cite{Paterson}).
 Observe that a discrete group $G$ has a nuclear reduced group $C^*$-algebra only if $G$ is amenable
 (see  \cite[Theorem 2.6.8]{BrownOzawa}).
 
\noindent
 (ii) Let $G$ be a simple non-compact Lie group of real rank 1, as for instance $SO(n,1)$ for $n\geq 2$ or $SU(n,1)$
for $n\geq 1.$ Then $G$ does not have Property $(T)$. 
However, as $C^*_r(G)$ is nuclear, it follows from Theorem~\ref{thm:2} that $C^*_r(G)$ has strong Property $(T)$.
 \end{example}
 
 In combination with a result from  \cite{Ng-str-amen}, Property (T) for  $C^*_r(G)$ admits the 
following characterization when  $G$ is a non-compact group. 
\begin{cor}
\label{cor-thm:2}
Let $G$ be a non-compact locally compact group such that  $C^*_r(G)$ is a nuclear $C^*$-algebra.
 The following statements are equivalent. 
	\begin{enumerate}
		\item $G$ is non-amenable;
		\item $C^*_r(G)$ admits no tracial state;
		\item $C^*_r(G)$ has Property $(T)$ (equivalently, strong Property $(T)$).
	\end{enumerate}
\end{cor}

Let $G$ be a non-amenable locally compact group  with a nuclear reduced $C^*$-algebra, as in Theorem~\ref{thm:2}. 
The proof that $C^*_r(G)$ has strong Property $(T)$ 
relies simply on  the fact  that there is no non-degenerate Hilbert $^*$-bimodule over $C^*_r(G)$ with approximate central vectors.

Apart from amenable groups, the only groups for which we are able to construct appropriate non-degenerate Hilbert $^*$-bimodules over $C^*_r(G)$  with approximate central vectors are the so-called IN-groups. 
Recall that $G$ is said to be an \emph{IN-group}  if  there exists  a compact neighborhood of the identity
in $G$ which is invariant under conjugation (for more detail on these groups, see \cite{Palmer}).
For this class of groups, we indeed do have a positive result concerning the relation between Property $(T)$ of $G$ and strong Property $(T)$ of $C^*_r(G)$.

 \begin{thm}
 \label{thm:main}
Let $G$ be an IN-group and $H\subseteq G$  a closed subgroup. 

\noindent
(a) If $\big(C^*_r(G), C^*_r(H)\big)$ has strong Property $(T)$, then $(G, H)$ has Property $(T)$. 

\noindent
(b) In the case where $G$ is $\sigma$-compact, the pair $(G, H)$ has Property $(T)$ if and only if $\big(C^*_r(G), C^*_r(H)\big)$ has strong Property $(T)$. 

\noindent
(c) $G$ has Property $(T)$ if and only if $C^*_r(G)$ has strong Property $(T)$. 
\end{thm}

Finally, we point out that the classes of groups considered in Theorem~\ref{thm:main} (IN-groups)  and in Theorem~\ref{thm:2} (non-amenable groups with a nuclear reduced $C^*$-algebra) are disjoint.
\begin{prop}
\label{Prop-IN-Nuclear}
Let $G$ be an IN-group with a nuclear reduced $C^*$-algebra. Then $G$ is amenable.
\end{prop}

This paper is organized as follows. In Sections~\ref{S1}, \ref{S2} and \ref{S3}, we  recall some  preliminary facts
and  establish some  tools which are crucial for the proofs of our results.   In Section 6, we 
conclude the proofs  of  Theorems~\ref{thm:full}, \ref{thm:full-QG}, \ref{thm:2} and \ref{thm:main},
as well as Corollary~\ref{cor-thm:2} and Proposition~\ref{Prop-IN-Nuclear}.

\section{Property $(T)$ for non-unital $C^*$-algebras}
\label{S1}
We recall from \cite{Ng-nu-T} the notion of Property $(T)$ and of strong Property $(T)$ for pairs (or inclusions) of general $C^*$-algebras. 

Let $A$ be a (not necessarily unital) $C^*$-algebra and  let $M(A)$ denote its mutiplier algebra. 
Let $\KH$ be a non-degenerate Hilbert $^*$-bimodule over $A$. 
Observe that the associated $^*$-representation and $^*$-anti-representation of $A$ (on $\KH$) uniquely extend to a $^*$-representation and a $^*$-anti-representation of $M(A)$ on the same space.

A net $(\xi_i)_{i\in \KI}$ in $\KH$ is said to be 
\begin{itemize}
	\item \emph{almost-$A$-central} if 
	$$\| a \cdot \xi_i - \xi_i \cdot a\| \to_{i} 0\qquad \text{ for every } a\in A;$$

	\item \emph{almost-$\KK_A$-central} if 
$$\sup_{x\in L} \| x \cdot \xi_i - \xi_i \cdot x\| \to_{i} 0$$ for every subset $L$ of $M(A)$ which is compact for the \emph{strict topology} on $M(A)$; recall that this is the weakest topology on $M(A)$ for which the maps $x\mapsto x a$ and $x\mapsto ax$ from $M(A)$ to $A$ are continuous for every $a\in A$, when  $A$ is equipped with the norm topology. 
\end{itemize}

Let $B$ be a $C^*$-subalgebra of $M(A)$ which is \emph{non-degenerate}, in the sense that an approximate identity of $B$ converges strictly to the identity of $M(A)$. 
The pair  $(A,B)$ is said to have \emph{Property $(T)$} (respectively, \emph{strong Property $(T)$}) if for every non-degenerate Hilbert $^*$-bimodule $\KH$ over $A$ with an almost-$\KK_A$-central net $(\xi_i)_{i\in \KI}$ of unit vectors, 
the space  $$\KH^B:=\{\xi\in \KH: b\cdot \xi = \xi \cdot b \quad \text{ for every }\quad b\in B \}$$ 
of \emph{central vectors} is non-zero (respectively, one has $\|\xi_i - P^B \xi_i\|\to_{i} 0$, where $P^B$ is the orthogonal projection onto the closed subspace $\KH^B$). 
 
 We say that $A$ has \emph{Property $(T)$} if $(A,A)$ has Property $(T)$ and that  $A$ has  \emph{strong Property $(T)$} 
if $(A,A)$ has strong Property $(T)$.

Notice that, compared to the original definition of Property $(T)$ for unital $C^*$-algebras in \cite{B}, 
we use here almost-$\KK_A$-central nets of unit vectors  instead of almost-$A$-central ones.
Taking $a_0=1$ in Item (b) of the following elementary lemma,  we see that these two definitions coincide in the unital case.
This elementary lemma will also be needed later on. 

\begin{lem}\label{lem:st-cpt}
Let $K$ be a non-empty strictly compact subset of $M(A).$

\noindent
(a) The set $K$ is norm-bounded. 

\noindent
(b) For any $a_0\in A$ and $\epsilon > 0$, one can find $x_1,\dots, x_n\in K$ such that for every $x\in K$, there exists $k\in \{1,\dots, n\}$ with $\|xa_0 - x_ka_0\| < \epsilon$. 
\end{lem}
\begin{proof}
Item (a) follows from the uniform boundedness principle; Item (b) is a consequence 
of the  fact that $Ka_0$ is norm-compact.
\end{proof}

\section{Hilbert $^*$-bimodules over nuclear reduced group $C^*$-algebras}

Let $G$ be a locally compact group.
Recall that, given a unitary representation $(u,  \KH)$ of $G$,  a net 
$(\xi_i)_{i\in \KI}$ of unit vectors in $\KH$ is \emph{almost-invariant} if  
$$\sup_{s\in Q} \| u(s)\xi_i - \xi_i \| \to_{i} 0$$ for every compact subset $Q$ of $G$.

Identifying every element $s\in G$ with the Dirac
measure $\delta_s,$ we view $G$ as a subset of $M(C^*(G))$. 
Notice that this embedding $\delta:s\mapsto \delta_s$ is continuous, when $M(C^*(G))$ is equipped with the strict topology.

Let $\KH$ be a non-degenerate Hilbert $^*$-bimodule over $C^*(G)$. 
Define a map $u_\KH:G\times G\to \mathcal{L}(\KH)$ by  
$$u_\KH(s,t)\xi := \delta_s\cdot \xi \cdot \delta_{t^{-1}} \qquad \text{for all} \quad s,t\in G, \xi\in \KH. $$ 
It is easily checked that  $u_\KH$ is a unitary representation of the cartesian product group $G\times G$. 
Thus, it induces a representation 
$$\tilde u_\KH: C^*(G\times G) \to \CB(\KH).$$ 

As we now see, when $C_r^*(G)$ is nuclear and $\KH$ comes from a $^*$-bimodule over $C^*_r(G)$, the representation $u_\KH$ factors through the regular representation of $G\times G$.

\begin{lem}
 \label{Lem1}
Let $\KH$ be a non-degenerate Hilbert $^*$-bimodule over $C^*(G)$
and $u_\KH$ the associated  unitary representation of $G\times G$. 

\noindent
(a) If $(\xi_i)_{i\in \KI}$ is an almost $\mathfrak{K}_{C^*(G)}$-central net of unit vectors, then  $(\xi_i)_{i\in \KI}$ is almost $u_\KH|_{\Delta(G)}$-invariant, where $\Delta(G):=\{(g,g)\mid g\in G\}$ is the diagonal subgroup of $G\times G$. 

\noindent
(b) Assume that $C_r^*(G)$ is nuclear.
Then $\KH$ is a Hilbert $^*$-bimodule over $C^*_r(G)$ (i.e., the associated $^*$-representation and $^*$-anti-representation of $C^*(G)$ both 
factorize
through $C^*_r(G)$) if and only if $u_\KH$ is weakly contained in the left regular representation $\lambda_{G\times G}$ of $G\times G$. 
\end{lem}
\begin{proof}
(a) Let $Q$ be  a compact subset of $G$.
Then $Q$ can be viewed as a strictly compact subset of $M(C^*(G))$ because $\delta$ is continuous. 
The claim follows, since  $\delta_s$ is a unitary operator for every $s\in G.$

\noindent
(b) Assume that $\KH$ is a Hilbert $^*$-bimodule over $C^*_r(G)$. 
By the universal property, there is a $^*$-representation $\phi_\KH$ of  the maximal tensor product $C^*_r(G)\otimes_{\rm {max}}C^*_r(G)^\mathrm{op}$ on $\KH$ compatible with the Hilbert $^*$-bimodule structure over $C^*_r(G)$ (here, $C^*_r(G)^\mathrm{op}$ means the opposite $C^*$-algebra of $C^*_r(G)$; see e.g. \cite{Ng-nu-T}). 
However, since $C^*_r(G)$ is nuclear, 
 $C^*_r(G)\otimes_{\rm {max}}C^*_r(G)^\mathrm{op}$ coincides canonically with the minimal tensor product $C^*_r(G)\otimes_{\rm{min}} C^*_r(G)^\mathrm{op}$. 

Now, under the canonical identifications of $C^*_r(G)$ with  $C^*_r(G)^\mathrm{op}$ (which, in the $L^1(G)$ level, sends $f$ to $(\bar f)^*$, with $\bar f$ being the complex conjugate of the function $f$) as well as the identification 
$$C^*_r(G)\otimes_{\rm{min}} C^*_r(G) \cong C^*_r(G\times G),$$
one has $\phi_\KH \circ \lambda = \tilde u_\KH$, where $\lambda$ is the canonical map from $C^*(G\times G)$ to $C^*_r(G\times G)$. 
In other words, $u_\KH$ is weakly contained in $\lambda_{G\times G}$.

Conversely, assume that $u_\KH$ is weakly contained in $\lambda_{G\times G}.$ Then, by definition of $u_\KH$, we see that
the associated $^*$-representation and $^*$-anti-representations of $C^*(G)$ on  $\KH$ factor through $C^*_r(G)$. 
\end{proof}

\section{On  the spectrum of a $C^*$-algebra with Property $(T)$}
\label{S2}
Let $A$ be a (not necessarily unital) $C^*$-algebra.
As said in the above, every non-degenerate $^*$-representation $\pi:A\to \CB(\KH)$ 
of $A$ extends canonically  to a non-degenerate $^*$-representation of $M(A)$, which will again be denoted by $\pi$. 
Moreover,  if $B$ is a non-degenerate $C^*$-subalgebra of $M(A),$  the restriction $\pi|_B$ of $\pi$ to $B$ is a non-degenerated representation of $B$.
Analogous statements are true for a non-degenerate $^*$-anti-representation. 

The following isolation property of  one dimensional $^*$-representations of $A$ is the crucial tool for the proof of 
Theorems~\ref{thm:full} and ~\ref{thm:full-QG}.
Concerning general facts about the topology of the spectrum (or dual space) $\widehat{A}$ of $A$, see Chapter~3 in
\cite{Dixm--C*}. 

\begin{prop}
\label{prop:prop-T>isolate}
Let $A$ be a $C^*$-algebra and $B$ be a non-degenerate $C^*$-subalgebra of $M(A)$. 
Let $\chi:A\to \BC$ be a non-zero $^*$-homomorphism.

\noindent
(a) Suppose that $(A,B)$ has Property $(T)$. 
For any non-degenerate $^*$-representation $(\pi, \KH)$ of $A$ which weakly contains $\chi$, the representation $\chi|_B$ is contained in $(\pi|_B, \KH)$. 

\noindent
(b) If $A$ has Property (T), then $\chi$ an isolated point in the spectrum $\widehat{A}$ of $A$.
\end{prop}
\begin{proof}
Notice that Item (b) follows from Item (a), by considering $A=B$ and $(\pi, \KH)\in \widehat{A}$. 
Hence, it suffices to prove Item (a).

By the assumption, there exists a net $(\xi_i)_{i\in \KI}$ of unit vectors
	in $\KH$ such that 
	\begin{equation}\label{eqt:limit-phi}
	{\lim}_i \Vert \pi(a)\xi_i-\chi(a)\xi_i\Vert =0 \qquad \text{for all} \quad a\in A.
	\end{equation}
	We have to prove that $\pi|_B$ actually contains $\chi|_B$.

	Define a Hilbert $^*$-bimodule structure (over $A$) on $\KH$ by 
	$$
	a\cdot \xi= \pi(a)\xi \qquad\text{and}\qquad \xi\cdot a= \chi(a)\xi  \qquad \text{for all} \quad a\in A, \xi\in \KH.
	$$
	Fix an element $a_0\in A$ with $|\chi(a_0)| = 1$. 
	It follows from Relation \eqref{eqt:limit-phi}  that 
	\begin{equation}\label{eqt:lim-phi-xi}
	{\lim}_i\Vert \pi(a_0)\xi_i\Vert = |\chi(a_0)| = 1. 
	\end{equation}
	Thus, without loss of generality, we may assume that $\pi(a_0)\xi_i\neq 0$ for all $i\in \KI$, and 
	we set 
	$$\eta_i:=\dfrac{\pi(a_0)\xi_i}{\Vert \pi(a_0)\xi_i\Vert}.$$ 
	
	Let $K$ be a non-empty  strictly compact subset of $M(A)$ and let $\epsilon>0$.
	By Lemma \ref{lem:st-cpt}, 
	$$
	C:={\sup}_{y\in K} \Vert y\Vert <\infty$$
	and there exist $x_1,\dots, x_n\in K$ such that
	for every $x\in K$, we can find $k\in  \{1,\dots, n\}$ with 
	\begin{equation}\label{eqt:x-k}
	\Vert xa_0 -x_ka_0\Vert \leq  \epsilon/12.
	\end{equation}
	Moreover, Relations \eqref{eqt:limit-phi} and \eqref{eqt:lim-phi-xi} produce $i_0\in \KI$ such that for $i\geq i_0$ and $l\in \{1,\dots,n \}$, one has
	\begin{equation}\label{eqt:*}
	\frac{1}{\Vert \pi(a_0)\xi_{i}\Vert}\leq 2, \quad 
	\Vert \pi(a_0)\xi_{i}-\chi(a_0)\xi_{i}\Vert \leq  \frac{\epsilon}{4 C} \quad 
	\text{and} \quad \Vert \pi(x_la_0)\xi_{i}- \chi(x_la_0)\xi_{i}\Vert \leq  \dfrac{\epsilon}{12}
	\end{equation}
		
	Let $x\in K$. Choose an integer $k$ in $\{1,\dots, n\}$ such that Relation \eqref{eqt:x-k} holds. 
	We then have, using Relations \eqref{eqt:x-k} and \eqref{eqt:*}, that 
	\begin{align*}
	\Vert x\cdot \eta_i- \eta_i\cdot x\Vert
	\ =\ & \frac{1}{\Vert \pi(a_0)\xi_i\Vert} \Vert \pi(x)\pi(a_0)\xi_i- \chi(x)\pi(a_0)\xi_i\Vert\\
	\ \leq\ & 2\Vert\pi(xa_0)\xi_i-\pi(x_ka_0)\xi_i\Vert+ 2 \Vert\pi(x_ka_0)\xi_i-\chi(x_ka_0)\xi_i\Vert+\\
	\ &\qquad \quad 2\Vert \chi(x_ka_0)\xi_i- \chi(xa_0)\xi_i\Vert+ 2|\chi(x)|\Vert\chi(a_0)\xi_i- \pi(a_0)\xi_i\Vert\\
	\ \leq \ & \epsilon/6 + \epsilon/6 + \epsilon/6 + \epsilon/2
	\end{align*}
	This shows that $(\eta_i)_{i\in \KI}$ is an almost $\KK_A$-central net of unit vectors in $\KH$. 
	
	Since, by the assumption,  $(A,B)$ has Property (T),  it follows that $\KH^B\neq \{0\}$. 
	Hence, there exists a non-zero vector $\xi\in \KH$ such that 
	$$
	\pi(b)\xi= \chi(b) \xi \qquad \text{for all } \quad b\in B.
	$$
	This means that $\chi|_B$ is contained in the representation $\pi|_B$.  
\end{proof}

\begin{remark}
As is well-known
(see Theorem~1.2.5 in \cite{BHV}),
 Property $(T)$ of a locally compact group $G$
can be characterized  by the fact that one (or equivalently, any)  finite dimensional irreducible representation
of $G$  is an isolated point in the dual space  of $G$. It is natural to ask whether Proposition~\ref{prop:prop-T>isolate} remains true when   $\chi$ is an arbitrary  finite dimensional representation of  the $C^*$-algebra $A$.
We will not deal with this question in the current paper. 
\end{remark}

\section{A bimodule over $C_r^*(G)$ associated to a unitary representation}
\label{S3}
Let  $G$ be a locally compact group. We will need to associate to every unitary representation  of $G$ a bimodule of $C^*_r(G),$
via a standard procedure (see \cite{ConnesJones}, \cite{B}).

As usual, we denote by $L^p(G)$ the Banach space $L^p(G,\mu)$
for a fixed Haar measure $\mu$ on $G$.  

\begin{lem}
\label{lem:rep-and-inv-vect}
Let $G$ be a locally compact unimodular group. 
Let $s\mapsto u_s$  be a unitary representation of $G$
on a Hilbert space $\KH.$ 

\noindent
(a) The space $\KH\otimes L^2(G)=L^2(G;\KH)$ is a non-degenerate Hilbert $^*$-bimodule over $C^*_r(G)$ for the following 
left and right actions:
$$(g\cdot \eta)(t):= \int_G g(s)\eta(s^{-1}t)\ \!d\mu(s) \quad \text{and} \quad
(\eta\cdot g)(t):= \int_G g(r^{-1}t)u_{t^{-1}r}(\eta(r))\ \! d\mu(r),$$
where $g\in L^1(G), \eta\in L^2(G;\KH)$ and $t\in G$. 

\noindent
(b) Let $H$ be a closed subgroup of $G$ and $\zeta\in (\KH\otimes L^2(G))^{C^*_r(H)}$. Then, for every $s\in H$, 
we have 
$$\mu\big(\{t\in G: u_s(\zeta(t)) \neq \zeta(sts^{-1})\}\big) = 0.$$
\end{lem}
\begin{proof}
(a) As the left-hand displayed equality is given by the left regular representation, it defines a $^*$-representation of $C^*_r(G)$. 
On the other hand, the right-hand displayed equality is given by the conjugation of the right regular representation by the ``Fell unitary'' $U\in \CB(L^2(G;\KH))$, where 
$$U(\eta)(t):= u_t(\eta(t)) \qquad (\eta\in L^2(G;\KH); t\in G).$$
Thus, it defines a $^*$-anti-representation of $C^*_r(G)$. 
It is easy to check that this $^*$-representation and this $^*$-anti-representations are non-degenerate, and they commute with each other. 

\noindent
(b) As noted in the above, $\KH\otimes L^2(G)$ is a unital Hilbert $^*$-bimodule over $M(C^*_r(G))$. 
Moreover, $\zeta$ is a $M(C^*_r(H))$-central vector. 
If $\delta_s$ is the canonical image of $s$ in $M(C^*_r(G))$, then 
$$\delta_s\cdot \zeta = \zeta \cdot \delta_s,$$
and the claim follows.
\end{proof}

\section{Proofs of the results}
\label{S4}
\subsection{Proof of Theorems~\ref{thm:full} and \ref{thm:full-QG}}
\label{S4-1}
 
To show  Item (a) of Theorem~\ref{thm:full}, assume that the pair $(C^*(G),C^*(H))$ has Property $(T)$. It follows  from Proposition \ref{prop:prop-T>isolate}(a) that $(G,H)$ has Property (T), by the definition of Property (T) for pairs of groups.
 Conversely, the fact  that  Property (T) for $(G,H)$ implies Property (T) for $(C^*(G),C^*(H))$
 was proved in \cite[Proposition 4.1(a)]{Ng-nu-T}. 
 
  Item (b) of Theorem~\ref{thm:full} follows from  Item (a) in combination with  \cite[Proposition 4.1(c)]{Ng-nu-T}.
 
 Theorem~\ref{thm:full-QG} follows from Proposition \ref{prop:prop-T>isolate}(b) and \cite[Proposition 3.2]{CN}. 
 
\subsection{Proof of Theorem~\ref{thm:2}}
  \label{S4:2}
 Let $\KH$ is a non-degenerate Hilbert $^*$-bimodule over $C^*_r(G)$.
Let $u_\KH$ and $\Delta(G)$ be as in Lemma \ref{Lem1}. 
Then $u_\KH$ is weakly contained in $\lambda_{G\times G}$ (by Lemma~\ref{Lem1}(b)). 
Since $\lambda_{G\times G}|_{\Delta(G)}$ is weakly contained in the left regular representation $\lambda_{\Delta(G)}$  of $\Delta(G)$ (see \cite[Proposition F.1.10]{BHV}), 
it follows that  there is no almost $u_\KH|_{\Delta(G)}$-invariant net of unit vectors, since ${\Delta(G)}\cong G$ is non-amenable. 
Hence, $\KH$ has no almost $\mathfrak{K}_{C^*_r(G)}$-central net of unit vectors,
by Lemma~\ref{Lem1}(a). It follows that $C^*_r(G)$ has strong Property (T).

\subsection{Proof of Corollary~\ref{cor-thm:2} }
  \label{S4:3}
  Under the assumption that $C^*_r(G)$ is nuclear, the equivalence of (1) and (2) follows from \cite[Theorem 8]{Ng-str-amen}. 
By Theorem~~\ref{thm:2}, we know that (1) implies (3).
Finally, assume that $G$ is amenable. 
Since, by assumption, $G$ is not compact, it follows from Theorem~\ref{thm:full} (see also the comments after 
the statement of Theorem \ref{thm:full-QG}) that $C^*_r(G)$ does not have Property $(T).$ 
Hence, (3) implies (1), 
and the proof is complete.

\subsection{Proof of Theorem~\ref{thm:main}}
\label{S4-4}
We fix a conjugation invariant compact neighborhood $V$ of the identity $e$ of the IN-group $G$ and choose a Haar measure $\mu$ on $G$. 
Let us recall the following two well-known facts:
\begin{itemize}
	\item $G$ is unimodular. 
	
	\item the characteristic function $\chi_V$ is in the center of the algebra $L^1(G)$. 
\end{itemize}

\noindent
(a) Notice that as $V$ is conjugation invariant, we have 
$$\chi_V(s^{-1}t) = \chi_V(ts^{-1}) \qquad (s,t\in G).$$ 
Let $(u,\KH)$ be a unitary representation of $G$ and $(\xi_i)_{i\in \KI}$ be an almost $u$-invariant net of unit vectors in $\KH$. 
We consider $\KH\otimes L^2(G) = L^2(G;\KH)$ to be a Hilbert $^*$-bimodule over $C^*_r(G)$, as in Lemma \ref{lem:rep-and-inv-vect}(a). 

By the assumption, for any $\epsilon > 0$ and any continuous function $g$ on $G$ with its support, $\supp g$, being compact, there exists $i_0\in \KI$ such that  for any $i\geq i_0$, one has $\sup_{s\in \supp g} \|\xi_i - u_{s^{-1}}\xi_i\| < \epsilon$. 
Hence, for any $i\geq i_0$, 
\begin{align*}
\|g\cdot (\xi_i \otimes \chi_V) &- (\xi_i \otimes \chi_V)\cdot g\|_{L^2(G;\KH)}^2\\ 
& = \int_G \Big\|\int_G g(s) \big(\chi_V(s^{-1}t) \xi_i - \chi_V(ts^{-1})u_{s^{-1}}\xi_i\big)\ \! ds\Big\|_\KH^2 \ \! dt \\
& \leq \int_G \Big(\int_{\supp g} |g(s)|\|\xi_i - u_{s^{-1}}\xi_i \|\chi_V(ts^{-1})\ \! ds\Big)^2 \ \! dt\\
& \leq \|g\|_{L^2(G)}^2 \epsilon^2 \int_G \int_{\supp g} \chi_V(ts^{-1})\ \! ds \ \! dt\\
& \leq \mu(\supp g)\mu(V)\|g\|_{L^2(G)}^2 \epsilon^2
\end{align*}
Thus, by an approximation argument,  it follows that the bounded net $(\xi_i \otimes \chi_V)_{i\in \KI}$ is almost-$C^*_r(G)$-central.

We consider $\epsilon >0$ and denote $\kappa:=\|\chi_V^2\|_{L^2(G)}$, where $\chi_V^2$ is the convolution product of $\chi_V$ with itself. 
Suppose that $K\subseteq M(C^*_r(G))$ is a strictly compact subset. 
Then by Lemma \ref{lem:st-cpt}(b) and the above, there exists $i_1\in \KI$ such that for any $i\geq i_1$, one has
\begin{equation}\label{eqt:K-chi-V-alm-cent}
{\sup}_{y\in K} \|y\chi_V \cdot (\xi_i \otimes \chi_V) - (\xi_i \otimes \chi_V) \cdot y\chi_V\|_{L^2(G;\KH)} < \kappa \epsilon.
\end{equation}
Moreover, one can find $i_2\geq i_1$ such that for every $i\geq i_2$, 
\begin{equation}\label{eqt:chi-V-alm-cent}
\|\chi_V \cdot (\xi_i \otimes \chi_V) - (\xi_i \otimes \chi_V) \cdot \chi_V\|_{L^2(G;\KH)} < \kappa\epsilon. 
\end{equation}
Set $h:= \dfrac{\chi_V^2}{\kappa}$.
Then 
$$\chi_V\cdot \frac{\xi_i \otimes \chi_V}{\kappa} = \xi_i \otimes h.$$
Inequalities  \eqref{eqt:K-chi-V-alm-cent} and \eqref{eqt:chi-V-alm-cent}, together with the fact that $\chi_V$ is in the center of $C^*_r(G)$, imply that
$$\|y\cdot (\xi_i \otimes h) - (\xi_i \otimes h)\cdot y\|_{L^2(G;\KH)} \leq (1+\|y\|)\epsilon \qquad \text{for all} \quad y\in K.$$
As $K$ is norm-bounded (by Lemma \ref{lem:st-cpt}(a)), we conclude that 
$(\xi_i \otimes h)_{i\in \KI}$ is an almost-$\KK_{C^*_r(G)}$-central net of unit vectors. 
Therefore, strong Property $(T)$ of $\big(C^*_r(G), C^*_r(H)\big)$ implies that
\begin{equation}\label{eqt:conseq-str-T}
\big\|\xi_i \otimes h - P^{C^*_r(H)}(\xi_i \otimes h)\big\|_{L^2(G;\KH)} \to_{i} 0. 
\end{equation}

For each $i\in \KI$, put 
$$\eta_i:=P^{C^*_r(H)}(\xi_i \otimes h)\in L^2(G;\KH)^{C^*_r(H)}$$ and let $\zeta_i$ 
 be the restriction of $\eta_i$ on $V^2:=\{st:s,t\in V\}$. 
As $h$ vanishes outside $V^2$, it follows from \eqref{eqt:conseq-str-T} that 
\begin{equation}\label{eqt:xi-otimes-h-close-to-zeta}
\|\xi_i \otimes h - \zeta_i\|_{L^2(G;\KH)} \to_{i} 0. 
\end{equation}
For a fixed $s\in H$, the condition $\eta_i\in L^2(G;\KH)^{C^*_r(H)}$ implies that 
$$u_s(\eta_i)(t) = \eta_i(sts^{-1})\quad \text{for } \mu\text{-almost every}\quad t\in G,$$ 
by Lemma \ref{lem:rep-and-inv-vect}(b).
Therefore, the  invariance of $V^2$ under conjugation ensures that 
\begin{equation*}\label{eqt:G-inv}
u_s(\zeta_i(t)) = \zeta_i(sts^{-1}) \quad \text{ for } \mu\text{-almost all }t\in V^2.
\end{equation*}
This, together with the inequality 
$$\int_{V^2} \|\zeta_i(t)\|\ \!dt  \leq  \mu(V^2)^{1/2} \|\zeta_i\|_{L^2(G;\KH)},$$
 implies that  $\int_{V^2} \zeta_i(t)\ \!dt$ exists and lies inside the space,  $\KH^{u|_H}$, 
 of $H$-invariant vectors in $\KH.$ 

Finally, it follows  from  
$$\int_{V^2} \|(\xi_i\otimes h - \zeta_i)(t)\|\ \!dt \leq \mu(V^2)^{1/2} \|\xi_i\otimes h - \zeta_i\|_{L^2(G;\KH)}$$
and from  \eqref{eqt:xi-otimes-h-close-to-zeta} that 
$$\Big\|\frac{\mu(V)^2}{\kappa} \xi_i - \int_{V^2} \zeta_i(t)\ \!dt\Big\|_\KH = \Big\|\int_{V^2} \xi_i\otimes h(t) - \zeta_i(t)\ \!dt\Big\|_\KH\to_i 0.$$ 
This finishes the proof  of Item (a).

\noindent
(b) The claim follows from Item (a) and \cite[Proposition 4.1(b)]{Ng-nu-T}.

\noindent
(c) The claim follows from Item   (a) and \cite[Proposition 4.1(c)]{Ng-nu-T}.

\begin{remark}
(i) If one relaxes the assumption of Theorem \ref{thm:main} to $\big(C^*_r(G), C^*_r(H)\big)$ having Property $(T)$, then the proof of Theorem \ref{thm:main} still produces a non-zero element $\eta\in L^2(G;\KH)^{C^*_r(H)}$, but we do not know whether the 
$H$-invariant vector $\int_{V^2} \eta(t)\ \!dt$ is non-zero. 

\noindent
(ii) Let us say that  a net $(\xi_i)_{i\in \KI}$ is \emph{almost-$\KK^0_A$-central} if $\sup_{x\in K} \| x \cdot \xi_i - \xi_i \cdot x\| \to_i 0$ for any subset $K\subseteq M(A)$ such that $Ka$ is  norm-compact for every $a\in A$. 
If we define weaker versions of Property $(T)$ and strong Property $(T)$ with almost-$\KK_A$-central nets of unit vectors being replaced by almost-$\KK^0_A$-central nets of unit vectors, then - as the proofs show - all the results in the article remain true. 
\end{remark}

\subsection{Proof of Proposition~\ref{Prop-IN-Nuclear}}
  \label{S4:5}
Let $V$ and $\chi_V$ be as in the proof of Theorem \ref{thm:main}. 
Consider $L^2(G)$ as a Hilbert $^*$-bimodule over $C^*_r(G)$ (see Lemma \ref{lem:rep-and-inv-vect}). 
Since $\chi_V$ is in the center of $L^1(G)$, we know $\chi_V$ is a $C^*_r(G)$-central vector in $L^2(G)$. 
Thus, $\rho\in C^*_r(G)^*_+$ defined by $\rho(x):= \langle x(\chi_V), \chi_V\rangle$ ($x\in C^*_r(G)$) is tracial. 
The conclusion now follows from Corollary \ref{cor-thm:2} (note that compact groups are amenable).

%\section*{Acknowledgement}

%The second named author is supported by the National Natural Science Foundation of China (11471168).

\end{document}